\documentclass{amsart}
\usepackage{amsmath,amsfonts,amsthm}
\usepackage{amssymb,latexsym}
\usepackage[colorlinks]{hyperref}
\usepackage{graphicx}
\usepackage[all]{xy}
\usepackage{layout}

\usepackage[colorlinks]{hyperref}
\usepackage{graphicx}
\usepackage[all]{xy}
\usepackage{layout}

\begin{document}

	\theoremstyle{plain}
	\newtheorem{theorem}{Theorem}[section]
	\newtheorem{lemma}[theorem]{Lemma}
	\newtheorem{proposition}[theorem]{Proposition}
	\newtheorem{corollary}[theorem]{Corollary}
	
	\theoremstyle{definition}
	\newtheorem{comment}{Comment}[section]
	\newtheorem{definition}[theorem]{Definition}
	\newtheorem{example}[theorem]{Example}
	\newtheorem{counterexample}[theorem]{Counterexample}
	\newtheorem{problem}[theorem]{Problem}
	\newtheorem{note}[theorem]{Note}
	
	\theoremstyle{remark}
	\newtheorem*{notation}{Notation}
	\newtheorem*{convention}{Convention}
	\newtheorem*{remark}{Remark}
	\numberwithin{equation}{section}
	
	\newcommand{\W}{$A\otimes\mathcal{W}\ $}
	\newcommand{\A}{$\mathcal {A}$}
	\newcommand{\R}{$\mathcal {R}$}
	\newcommand{\M}{\mathbb M_{n}}
	\newcommand{\N}{\mathbb M_{N}}
	\newcommand{\NN}{\mathbb M_{N'}}
	\newcommand{\MMM}{\mathbb M_{k(n)}(\mathfrak A^{**})}
	\newcommand{\MM}{\mathbb M_{k(n)}}
	\newcommand{\C}{$C^*$}
	
	\newcommand{\Lrightarrow}{\hbox to1cm{\rightarrowfill}}
	\newcommand{\Ldownarrow}{\bigg\downarrow}

	\title[Fourier Algebra  of Fell Bundles]{Fourier and Fourier-Stieltjes Algebra  of Fell Bundles over discrete groups}
	
	\author[ M. Amini  \MakeLowercase{and} M. R. Ghanei ]{Massoud Amini
		\MakeLowercase{and} Mohammad Reza Ghanei}

	\address{Department of Mathematics, Khansar Campus, University of Isfahan,  Iran.}
	\email{\textcolor[rgb]{0.00,0.00,0.84}{m.r.ghanei@khn.ui.ac.ir ,
			mrg.ghanei@gmail.com}}
	
	\address{%
		Faculty of Mathematical Sciences\\ Tarbiat Modares University\\ Tehran 14115-134\\ Iran\newline
		School of Mathematics
		Institute for Research in Fundamental Sciences (IPM)\\ Tehran 19395-5746\\ Iran}
	
	\email{\textcolor[rgb]{0.00,0.00,0.84}{mamini@modares.ac.ir, mamini@ipm.ir}}

	
	
	
	
	
	
	
	\begin{abstract}
		For a Fell bundle $\mathcal{B}=\left\{B_{s}\right\}_{s \in G}$ over a discrete group $G$, we use representations theory of $\mathcal{B}$ to construct the 	Fourier and  Fourier-Stieltjes spaces $A(\mathcal{B})$ and  $B(\mathcal{B})$ of $\mathcal B$. When $\mathcal B$ is saturated we show  $B(\mathcal{B})$ is canonically isomorphic to
		the   dual space of the cross sectional $C^{*}$-algebra $C^{*}(\mathcal{B})$ of $\mathcal{B}$.
		When there is a compatible family of co-multiplications on the fibers we show that $B(\mathcal{B})$ and $A(\mathcal{B})$
		are Banach algebras. This holds in particular if either the fiber $B_e$ at identity is a Hopf $C^*$-algebra or $\mathcal{B}$ is the Fell bundle of a $C^*$-dynamical
		system.
		When $A(\mathcal{B})$ is a Banach algebra with bounded approximate identity,  we show that $B(\mathcal{B})$ is the
		multiplier algebra of $A(\mathcal{B})$. We prove a Leptin type  theorem by showing that amenability of $G$
		implies  the existence of bounded approximate identity for  $A(\mathcal{B})$ for bundles coming from a $C^*$-dynamical system $(A,G,\gamma)$. The converse is left as an open problem.

	\end{abstract} 
	
	\maketitle

	\section{Introduction}
	The Fourier algebra $A(G)$ and Fourier-Stieltjes algebra $B(G)$  of a  locally compact group $G$ are defined and studied by Pierre Eymard in \cite{Eymard64} and play an important role in modern theory of abstract harmonic analysis (see, for instance \cite{KaLa18}). The Fourier-Stieltjes algebra is the collection of all
	coefficient functions  of
	continuous unitary representations   of $G$, whereas the Fourier algebra $A(G)$
	consists of coefficient functions of the left regular representation. The Fourier-Stieltjes algebra $B(G)$ admits a
	natural norm making it into a commutative Banach $*$-algebra under pointwise multiplication.
	The Fourier
	algebra $A(G)$ as a closed ideal of $B(G)$. When $G$ is amenable, $B(G)$ is the multiplier algebra of $A(G)$; for more details, see \cite{KaLa18}

	Fell bundles were introduced in the late 60's \cite{Fell69} by James Fell to provide a framework for  understanding and extending George Mackey's theory of induced representations \cite{M68}.
	We refer the reader to \cite{FeDo88}  for the basic theory of Fell bundles and to \cite{Exel17} for a concise treatment and of
	this and other important concepts.
	
	Roughly speaking, a Fell bundle (also known as a $C^{*}$-algebraic bundle)
	over a discrete group $G$ is a collection
	$\mathcal{B}$ of closed subspaces of a $C^{*}$-algebra indexed by $G$, such that multiplication and involution in fibers follow the group rules. Now, just as in group $C^*$-algebras, one could use representation theory of the $*$-algebra of finitely supported sections of the bundle to associate full and reduced cross sectional $C^*$-algebras  $C^{*}(\mathcal{B})$ and  $C^{*}_r(\mathcal{B})$. For the trivial bundle (with fibers $\mathbb C$ everywhere) these are the same as full and reduced group $C^*$-algebras  $C^{*}(G)$ and  $C^{*}_r(G)$. Now since the Fourier-Stieltjes algebra $B(G)$ as a Banach space is nothing but the  dual space of the full group $C^*$-algebra $C^*(G)$, it is natural to expect that ``coefficients'' of representations of a Fell bundle $\mathcal B$ form a vector space, whose completion (in an appropriate norm) is identified with the dual space of $C^{*}(\mathcal{B})$. Moreover, in the group case, positive definite functions on the group canonically associate with positive functionals on $C^*(G)$, and so the next thing to look at would be the order structure on the positive cone of  $C^{*}(\mathcal{B})$ and the corresponding positive functionals.
	Indeed, Fell and Doran already initiated studying  properties of the positive cone of a Fell bundle in their original monograph \cite{FeDo88} and gave a ``module structure'' on (the linear span of) positive definite functionals on the bundle
	over (the linear span of) positive definite functions on the base group.
	
	The other problem to deal with is whether the  coefficients of representations of $\mathcal B$ form an algebra, or equivalently, whether the product of two positive definite functionals could be defined, and happens to be a positive functional (or at least inside their linear span).
	In the group case, the algebra structure on the dual $B(G)$ of full group $C^*$-algebra
	$C^*(G)$ comes from the co-algebra structure on $C^*(G)$ given by the canonical  co-product $\Delta:C^*(G)\rightarrow C^*(G)\otimes C^*(G)$. Such a co-algebra structure is not available for general Fell bundles, and could be expected only for Fell bundles associated to group actions on $C^*$-bi-algebras.
	
	In this paper we define and study Fourier and Fourier-Stieltjes spaces of a Fell bundle. We are particular interested in the spacial case where these are Banach algebras. In this latter case, we investigate the existence of a bounded approximate identity in the Fourier algebra of a Fell bundle and relate it to identifying its multiplier algebra. We study the special cases of bundles associated to group actions in somewhat more details. Similar construction is done by Bedos and Conti in \cite{BedCon16}.
	
	Most of the material of this paper could be extended (without considerable difficulty) to Fell bundles over locally compact groups, but we feel that at this stage (of developing a new theory), this level of generality is not necessary. Moreover, certain natural settings in which Fell bundles are motivated and discussed are only developed in the discrete case (see for instance \cite{Q96}).

	The paper is organized as follows. In Section \ref{2}, we give a brief overview of Fell bundles over
	discrete groups, fixing basic notations used  in the rest of this  paper.
	In Section \ref{3}, we define the Fourier-Stieltjes Space $B(\mathcal{B})$ of a
	Fell bundle $\mathcal{B}$ over a discrete group $G$. We show that $B(\mathcal{B})$ is a dual Banach space, and
	when $\mathcal B$ is saturated, we observe  that
	$B(\mathcal{B})$ is canonically isomorphic to the dual of cross sectional $C^*$-algebra $C^{*}(\mathcal{B})$ of $\mathcal{B}$.
	In Section \ref{4}, we find general conditions under which   $B(\mathcal{B})$ and  $A(\mathcal{B})$ are Banach algebras, and
	determine when the multiplier algebra of $A(\mathcal{B})$ is $B(\mathcal{B})$.

	\section{Preliminaries}\label{2}	
	Throughout this paper we work with a discrete group $G$, with unit element $e$. We denote  the set of all linear continuous maps between Banach spaces $X$ and $Y$ by $\mathcal{L}(X, Y)$. We denote $\mathcal{L}(X, X)$ by $\mathcal{L}(X)$.
	
	\begin{definition}
		A {\it Fell bundle} (or a $C^{*}$-algebraic bundle)
		over a discrete group $G$ is a collection
		$
		\mathcal{B}=\left\{B_{s}\right\}_{s \in G}
		$
		of closed subspaces of a $C^{*}$-algebra $B$, indexed by $G$,
		satisfying $B_{s} B_{t} \subseteq B_{s t}$ and $B_{s}^{*}=B_{s^{-1}}$, for  $s, t \in G$.	The disjoint union of fibers is called the {\it total space} of $
		\mathcal{B}$ (again denoted by $\mathcal{B}$). A {\it section} of $
		\mathcal{B}$ is a map $y$ from $G$ to the total space of $
		\mathcal{B}$ with $y_s:=y(s)\in B_s$, for $s\in G$.
	\end{definition}
	
	A Fell bundle
	$
	\mathcal{B}=\left\{B_{s}\right\}_{s \in G}$ is said to be {\it saturated} if the linear span $B_sB_t$ of all multiples of elements of $B_s$ in that of $B_t$ is dense in $B_{st}$, for each $s,t\in G$ \cite[VIII.2.8]{FeDo88}. For elementary examples of saturated bundles, see \cite[VIII.3.15-16]{FeDo88}.
	
	The set of all finitely supported (respectively, countably summable) sections of $
	\mathcal{B}$ is denoted by
	$c_c(\mathcal{B})$ (respectively, by $\ell^1(\mathcal{B})$). Taking supremum over all $C^*$-seminorms on $c_c(\mathcal{B})$ one gets a $C^*$-seminorm $\|.\|_{\rm max}$ on $c_c(\mathcal{B})$, and the completion of the quotient by norm zero elements is a $C^*$-algebra $C^*(\mathcal{B})$, called the {\it cross sectional} $C^*$-{\it algebra} of $\mathcal{B}$. By a similar construction using the {\it regular representation} of $\mathcal{B}$, one could construct the {\it reduced cross sectional} $C^*$-{\it algebra} $C^*_r(\mathcal{B})$ of $\mathcal{B}$; for more details see, \cite[17.6]{Exel17}.
	
	\begin{example}\label{bundle}
		Let us exemplify the above notions in certain basic cases.
		
		\vspace{.2cm}
		$(i)$ Let $G$ be a discrete group. Consider the trivial bundle over $G$, i.e.,
		$\mathcal{B}=G\times\mathbb{C}=\{B_s\}_{s\in G}$, where $B_s=\{(s,z)|z\in\mathbb{C}\}$, with operations
		$(s,z)(t,w)=(st,zw)$, $(s,z)^*=(s^{-1},\bar{z})$ and norm $\|(s,z)\|=|z|$.	
		Then canonically,  $c_c(\mathcal{B})\cong c_c(G)$, $\ell^1(\mathcal{B})\cong \ell^1(G)$,   $C^*_r(\mathcal{B})\cong C^*_r(G)$, and $C^*(\mathcal{B})\cong C^*(G)$.	
		
		\vspace{.2cm}
		$(ii)$ Let $(A,G,\gamma)$ be a $C^*$-dynamical system. For the fibers
		$B_s=\{s\}\times A$,	with operations $(s,a)(t,b)=(st,a\gamma_s(b))$, $(s,a)^*=(s^{-1},\gamma_{s^{-1}}(a^*))$,
		and norm $\|(s,a)\|=\|a\|$,  $c_c(\mathcal{B})\cong c_c(G, A)$, $\ell^1(\mathcal{B})\cong \ell^1(G, A)$,  $C^*_r(\mathcal{B})\cong A\rtimes_r G$, and $C^*(\mathcal{B})\cong A\rtimes G$.
		
		\vspace{.2cm}
		
		$(iii)$ A coaction of discrete group $G$ on a $C^*$-algebra $A$ is an injective non-degenerate homomorphism
		$\delta: A \rightarrow$ $A \otimes C^{*}(G)$
		such that $(\delta \otimes \mathrm{id}) \circ \delta=\left(\mathrm{id} \otimes \delta_{G}\right) \circ \delta$,
		where $\delta_{G}: C^{*}(G) \rightarrow C^{*}(G) \otimes C^{*}(G)$ is the canonical co-product defined by
		$\delta_{G}(u_s)=u_s \otimes u_s$, $s \in G$, where $u$ is the universal representation of $G$. With the spectral subspaces  $A_{s}:=\{a \in A: \delta(a)=a \otimes s\}$ as  fibers, $\mathcal A=\{A_s\}$ is a Fell bundle with $A\cong C^*(\mathcal A)$. Conversely, for each Fell bundle $\mathcal A=\{A_s\}$, $G$ coacts on $C^*(\mathcal A)$ via $$\delta: C^*(\mathcal A) \rightarrow C^*(\mathcal A) \otimes C^{*}(G);\ \  \delta(a)=a\otimes \delta_s,$$ for $a\in A_s$, where $\delta_s\in c_c(G)$ is regarded as an element in $C^*(G)$ (cf., \cite[3.3]{Q96}).

	\end{example}

	Given  sections $y$ and $z$ in $c_{c}(\mathcal{B})$, we define their convolution  by
	$$
	(y \star z)_{s}=\sum_{t \in G} y_{t} z_{t^{-1} s} \in B_{s},
	$$
	and the involution  by $y_{s}^{*}=\left(y_{s^{-1}}\right)^{*}$, for  $s \in G .$
	The space $c_{c}(\mathcal{B})$ is an associative $*$-algebra under these operations.

	\begin{definition}
		A \textit{representation} $\pi$ of a Fell bundle $\mathcal{B}=\left\{B_{s}\right\}_{s \in G}$ on
		a Hilbert space $H_\pi$ is
		a map $\pi: \mathcal{B} \rightarrow \mathcal{L}(H_\pi)$ such that, for each $s \in G$, the map
		$$
		\pi_{s}:=\pi(s): B_{s} \rightarrow \mathcal{L}(H_\pi)
		$$
		is linear on each fiber, and for all $s, t \in G, b \in B_{s}$ and $c \in B_{t}$,
		
		\vspace{.2cm}
		($i$) $\pi_{s}(b) \pi_{t}(c)=\pi_{s t}(b c)$,
		
		($ii$) $\pi_{s}(b)^{*}=\pi_{s^{-1}}\left(b^{*}\right)$.
		
		\vspace{.2cm}
		\noindent We always assume that $\pi: \mathcal{B} \rightarrow \mathcal{L}(H_\pi)$ is \textit{continuous}, that is, for every $s \in G$
		and $\xi\in H_\pi$,  the map
		$
		b\in B_{s} \mapsto \pi_s(b)\xi
		$
		is norm-continuous. Since,
		$$
		\|\pi_s(b)\|^2=\|\pi_{s^{-1}}(b^*)\pi_s(b)\|=\|\pi_e(b^*b)\|\leq\|b^*b\|=\|b\|^2,
		$$
		we get, $\|\pi_s(b)\|\leq\|b\|$.
	\end{definition}
	
	\begin{example} There are basic examples of representations.
		
		\vspace{.2cm}
		$(i)$ For the trivial bundle $\mathcal B=G\times\mathbb C$, each  representation $u$ of $G$ gives a representation  $\Pi=\{\pi_s\}$ of the trivial bundle, via $\pi_s(s,z)=u_s$. Conversely, given a representation $\Pi=\{\pi_s\}_{s\in G}$ of the trivial bundle,   $u_s=\pi_s(s,1)$ defines a representation of $G$.	
		
		\vspace{.2cm}
		$(ii)$ Given  a representation  $\Pi$ of the Fell bundle $\mathcal{B}$ of a
		$C^*$-dynamical system $(A,G,\gamma)$,   $\pi(a)=\Pi_e(a,e)$	
		defines a representation of $A$, and when  $A$ is unital,  $u(s)=\Pi_s(s,1)$ gives a representation of $G$.	Conversely, each covariant representation $(\pi, u)$ of $(A,G,\gamma)$ defines a representation $\Pi_s(a)=\pi(a)u_s$, for $s\in G, a\in A$.
	\end{example}
	
	For a representation $\pi$ of a Fell bundle $\mathcal{B}$, we consider the coefficient functions
	$$
	\pi_{\xi, \eta, s}: B_{s} \rightarrow \mathbb{C};\quad \pi_{\xi, \eta, s}(b)=\left\langle\pi_{s}(b) \xi, \eta\right\rangle,
	$$
	for  $\xi, \eta \in H_{\pi}$ and $s \in G$.
	Two representations $\pi$ and $\sigma$ of $\mathcal{B}$ are
	\textit{equivalent} if there exists a
	unitary  $U: \mathcal{H}_{\pi} \rightarrow \mathcal{H}_{\sigma}$ such that
	$U \pi_{s}(b) \xi=\sigma_{s}(b) U \xi$, for  $s \in G, b \in B_{s}$, and
	$\xi \in H_{\pi}$. We  write $\pi \sim \sigma$ in this case and call $U$ an
	\textit{intertwining} operator for $\pi$ and $\sigma$.
	We  denote the collection of equivalence classes of   representations
	of $\mathcal{B}$ by $\Sigma_{c}(\mathcal{B})$.
	
	For $s$ in $G$, we denote by $j_{s}$
	the natural inclusion of $B_{s}$ in $C_{c}(\mathcal{B})$, namely the map
	$
	j_{s}: B_{s} \rightarrow C_{c}(\mathcal{B}),
	$
	defined by
	$$
	\left.j_{s}(b)\right|_{h}= \begin{cases}b, & \text { if } h=s \\ 0, & \text { otherwise. }\end{cases}
	$$
	The inclusion  $j_{e}$ at identity is a $*$-homomorphism through which we may view $B_{e}$ as a subalgebra
	of $c_{c}(\mathcal{B})$. This also makes $c_{c}(\mathcal{B})$ a right $B_{e}$-module with a $B_{e}$-valued inner-product
	$$
	\langle y, z\rangle=\sum_{s \in G}y_{s}^{*}z_{s}.
	$$
	We  denote by $\ell^{2}(\mathcal{B})$ the right Hilbert $B_{e}$-module obtained by completing
	$c_{c}(\mathcal{B})$ under the norm $\|\cdot\|_{2}$ coming  from the above inner-product.
	
	We may regard  $j_{s}$ as a map from $B_{s}$ to $\ell^{2}(\mathcal{B})$,
	which is then an isometry.
	Given $b\in B_{s}$, the map
	$$
	\lambda_{s}(b): y \in c_{c}(\mathcal{B}) \longmapsto j_{s}(b) \star y \in c_{c}(\mathcal{B}),
	$$
	is continuous in the Hilbert module norm
	and extends to a bounded operator on $\ell^{2}(\mathcal{B})$, still
	denoted by $\lambda_{s}(b)$, such that, $\left\|\lambda_{s}(b)\right\|=\|b\|$,
	and
	$$
	\lambda_{s}(b)(j_{t}(c))=j_{s t}(b c),\ \left\langle\lambda_{s}(b) \xi, \eta\right\rangle=\left\langle\xi, \lambda_{s-1}\left(b^{*}\right) \eta\right\rangle \ ( t \in G, c \in B_{t},  \xi, \eta \in \ell^{2}(\mathcal{B})).
	$$
	The representation
	$
	\lambda=\left\{\lambda_{s}\right\}_{s \in G}
	$
	is
	called the
	\textit{regular representation} of $\mathcal{B}$.
	
	\section{Fourier and Fourier-Stieltjes spaces of Fell bundles} \label{3}
	
	Given a Fell bundle $\mathcal{B}=\left\{B_{s}\right\}_{s \in G}$ over a discrete group $G$, the \textit{$\ell^1$-cross sectional algebra}  $\ell^{1}(\mathcal{B})$,
	consists of the $\ell^1$-cross sections of $\mathcal{B}$,
	with operations
	$$
	f * g(t)=\sum_{s \in G} f(s) g(s^{-1} t), \ f^{*}(t)=f(t^{-1})^{*}, \quad  (t \in G, f, g \in \ell^{1}(\mathcal{B}))
	$$
	and norm
	$
	\|f\|_{1}=\sum_{s \in G}\left\|f_{s}\right\|$.
	To each $\pi \in \Sigma_{c}(\mathcal{B})$, one could associate an algebra representation
	$\tilde{\pi}: \ell^{1}(\mathcal{B}) \rightarrow \mathcal{L}(H_{\pi})$, given by
	$\langle\tilde{\pi}(f) \xi, \eta\rangle=\sum_{s \in G}\left\langle\pi_{s}(f_{s}) \xi, \eta\right\rangle$,
	such that
	$$
	\begin{aligned}
		|\langle\tilde{\pi}(f) \xi, \eta\rangle| & \leq\|\xi\|\|\eta\| \sum_{s \in G}\left\|\pi_{s}\left(f_{s}\right)\right\|
		\leq\|\xi\|\|\eta\|\|f\|_{1}.
	\end{aligned}
	$$
	
	For $f \in \ell^{1}(\mathcal{B})$, put
	$
	\|f\|_{*}=\sup \left\{\|\tilde{\pi}(f)\|: \pi \in \Sigma_{c}(\mathcal{B})\right\}.
	$
	For the coefficient function
	$u=\pi_{\xi,\eta}$,
	$$
	\big|\sum_{s\in G}u_s(f_s)\big|
	=\big|\sum_{s \in G}\langle\pi_s(f_s)\xi,\eta\rangle\big|
	=|\langle\tilde{\pi}(f)\xi,\eta\rangle|
	\leq \|f\|_*\|\xi\|\|\eta\|. 		
	$$
	
	This allows us to define a Banach space structure on the set of coefficient functions.

	\begin{definition}
		Let 	$\mathcal{B}=\left\{B_{s}\right\}_{s \in G}$ be  a Fell bundle over discrete group $G$.
		The \textit{Fourier-Stieltjes} space $B(\mathcal{B})$	of 	$\mathcal{B}$
		is the space
		$$
		B(\mathcal{B}):=\left\{\pi_{\xi, \eta}|
		~\pi\in\Sigma_{c}(\mathcal{B}),~~\xi,\eta\in H_\pi\right\},
		$$
		with  norm
		$
		\|u\|_{B(\mathcal{B})}=	\sup\{\big|\sum_{s\in G}u_s(f_s)\big|:f\in\ell^1(\mathcal{B}), \|f\|_*\leq 1\}.
		$

	\end{definition}
	
	For a representation $\pi$ of a Fell bundle $\mathcal{B}$ and
	a unitary representation $u$ of $G$, we may define a representation
	$$u\otimes\pi:\mathcal{B}\rightarrow\mathcal{L}(H_u\otimes H_\pi); \ (u\otimes\pi)_s(b)=u_s\otimes\pi_s(b), \ \ (s\in G, b\in B_s).$$ Similarly, we have a representation
	$\pi\otimes u:\mathcal{B}\rightarrow\mathcal{L}(H_\pi\otimes H_u)$.

	Let $\pi:\mathcal{B}\rightarrow\mathcal{L}(H_\pi)$ a  representation of $\mathcal{B}$ and
	$u$ be a unitary representation of $G$.
	For $\xi,\eta\in H_\pi$ and $\alpha,\beta\in H_u$, and functions $u_{\alpha,\beta}\in B(G)$ and
	$\pi_{\xi,\eta}\in B(\mathcal{B})$, put
	$$
	(u_{\alpha,\beta}\cdot\pi_{\xi,\eta})_s(b)=\langle(u_s\otimes\pi_s(b))(\alpha\otimes\xi),\beta\otimes\eta\rangle,
	\quad (s\in G, b\in B_s).
	$$
	Then $B(\mathcal{B})$ is a left $B(G)$-module.
	Similarly, using $\pi\otimes u$, $B(\mathcal{B})$ has a right $B(G)$-action.

	A functional $\phi$ on $C^*(\mathcal{B})$ is  positive if $\phi(u^**u)\geq 0$ for each
	$u\in C^*(\mathcal{B})$. For a representation $\pi$  in $ \Sigma_{c}(\mathcal{B})$ and
	$\xi\in H_\pi$, the functional $\tilde{\pi}_{\xi,\xi}:C^*(\mathcal{B})\rightarrow\mathbb{C}$ defined by
	$\tilde{\pi}_{\xi,\xi}(g)=\langle \tilde{\pi}(g)\xi,\xi\rangle$ is a positive linear functional on
	$C^*(\mathcal{B})$ and for each $f\in\ell^{1}(\mathcal{B})$, we have
	$\tilde{\pi}_{\xi,\xi}(f)=\sum_{r \in G}\langle\pi_{r}(f_r)\xi, \xi\rangle$.
	\begin{definition}
		An element $\varphi\in B(\mathcal{B})$	is called \textit{positive definite}  if
		there exists a representation $\pi\in\sum_{c}(\mathcal{B})$ and a vector $\xi\in H_\pi$ such that $\varphi=\pi_{\xi,\xi}$.	
	\end{definition}
	In \cite{FeDo88}, Fell and Doran defined the functionals of  positive type on $\mathcal{B}$ as functionals $P$  satisfying the inequality
	$\sum_{i,j} P(b_j^*b_i)\geq 0$, for any finite sequence $b_1,...,b_n\in\mathcal{B}$.
	By  \cite[20.6]{FeDo88}, our positive definite functions correspond to functionals of positive type in the sense  of Fell and Doran.
	We let $P(\mathcal{B})$ be the collection of all positive definite functions on $\mathcal{B}$.

	\begin{lemma}\label{eta2}
		Let $\mathcal{B}$ be a Fell bundle over a discrete group $G$ and  $\psi=\pi_{\eta,\eta}\in P(\mathcal{B})$, for
		some $\pi\in\Sigma_{c}(\mathcal{B})$ and $\eta\in H_\pi$. Then
		$\|\psi\|=\|\eta\|^2$.  	
	\end{lemma}
	\begin{proof}
		Let us denote the unit element of $C^*(\mathcal{B})^{**}$ by $1$ and choose a net $(g_\alpha)_\alpha$ in $ C^*(\mathcal{B})$ with
		$g_\alpha\rightarrow 1$ in the weak$^*$ topology of $C^*(\mathcal{B})^{**}$. Then,
		$\tilde{\pi}(g_\alpha)\rightarrow\mathrm{id}$ in the weak$^*$ toplogy of $\mathcal{L}(H_\pi)$.
		Use Hahn Banach to extend
		$\psi:C^*(\mathcal{B})\rightarrow\mathbb{C}$ to $\tilde{\psi}:C^*(\mathcal{B})^{**}\rightarrow\mathbb{C}$
		such that $\|\tilde{\psi}\|=\|\psi\|$.	Then,
		$$
		\tilde{\psi}(1)=\lim_\alpha\psi(g_\alpha)=\lim_\alpha\langle\tilde{\pi}(g_\alpha)\eta,\eta\rangle=
		\langle\eta,\eta\rangle=\|\eta\|^2.
		$$
		Therefore, $\|\tilde{\psi}\|\geq \|\eta\|^2\geq\|\psi\|$ and so $\|\psi\|=\|\eta\|^2$.
		
	\end{proof}

	\begin{proposition}\label{positivity}
		Let $\mathcal{B}$ be a saturated Fell bundle over a discrete group $G$. Then,
		
		\vspace{.2cm}
		$(i)$ if $\varphi$	is
		a positive linear functional  on $C^*(\mathcal{B})$, then $\varphi\in P(\mathcal{B})$.
		
		$(ii)$ every element of $B(\mathcal{B})$ is a finite linear combination of elements of $P(\mathcal{B})$.
		
		$(iii)$ there is an isometric isomorphism from $	B(\mathcal{B})$ onto $C^{*}(\mathcal{B})^*$.	
		
		$(iv)$ for $\varphi\in B(\mathcal{B})$,
		there exist a representation $\pi\in\Sigma_{c}(\mathcal{B})$ and  vectors $\xi,\eta\in H_\pi$	such that
		$
		\varphi=\pi_{\xi,\eta}$ and $\|\varphi\|=\|\xi\|\|\eta\|.
		$
		
		$(v)$ for
		$\varphi\in B(\mathcal{B})$,
		$
		\|\varphi\|=\inf\{\|\xi\|\|\eta\|~:~\varphi=\pi_{\xi,\eta}\}.
		$		
	\end{proposition}
	\begin{proof}
		$(i)$ Clearly $\varphi|_{\ell^{1}(\mathcal{B})}\in \ell^{1}(\mathcal{B})^*$
		is  positive  and   $\ell^{1}(\mathcal{B})^*\cong\ell^{\infty}(\mathcal{B}^*)$,
		where
		$$
		\ell^{\infty}(\mathcal{B}^*)=\{\psi:G\rightarrow\mathcal{B}^*:\  \psi(r)\in B_r^*, \|\psi\|_\infty:=\sup_{r\in G}\|\psi_r\|<\infty\}.
		$$
		For  $f,g\in \ell^{1}(\mathcal{B})$,  put
		$
		\langle f, g\rangle=\sum_{r\in G}\varphi_r\left((g^**f)_r\right)	
		$
		and observe that
		$
		|\langle f, g\rangle|\leq \|\varphi\|_\infty\|f\|_1\|g\|_1.
		$
		Let $\mathcal{N}=\{f\in\ell^{1}(\mathcal{B}): \langle f, f\rangle=0\}$.
		By  Schwartz inequality,  $\mathcal{N}$ is a left ideal in $\ell^{1}(\mathcal{B})$ and we get an inner product on the quotient space
		$\ell^{1}(\mathcal{B})/\mathcal{N}$. Denote the Hilbert space completion of
		$\ell^{1}(\mathcal{B})/\mathcal{N}$ by $H_\varphi$, and the image of
		$f\in \ell^{1}(\mathcal{B})$ in $ H_\varphi$ under the quotient map by $\tilde{f}$.
		
		Let $\psi_\alpha$ be an approximate identity in $ \ell^{1}(\mathcal{B})$, which exists as $\mathcal B$ is saturated \cite[VIII.2.13]{FeDo88}. Then
		for  $f\in \ell^{1}(\mathcal{B})$,
		$$
		\langle\tilde{f},\psi_\alpha\rangle=\sum_{s\in G}\varphi_s\left((\psi_\alpha^**f)_s \right)
		\rightarrow\sum_{s\in G}\varphi_s(f_s)=\varphi(f).
		$$
		Note that $\lim_\alpha\langle v, \psi_\alpha\rangle$ exists for each
		$v\in H_\varphi$, and hence  $\tilde{\psi_\alpha}$ converges weakly in $H_\varphi$
		to an element $\eta$ with
		$\langle \tilde{f},\eta\rangle=\varphi(f)$, for  $f\in \ell^{1}(\mathcal{B})$. Next, we define
		an algebra representation $$\rho:\ell^{1}(\mathcal{B})\rightarrow \mathcal{L}(H_\varphi); \ \rho(f)\tilde{g}=(f*g)^{\tilde{}},$$ for
		$f,g\in \ell^{1}(\mathcal{B})$, and put
		$\pi_s(b):=\rho(J_s(b))$, for $s\in G, b\in B_s.$
		Then, $\{\pi_s\}\in\Sigma_{c}(\mathcal{B})$ and
		$$
		\langle \tilde{g},\rho(f)\eta\rangle= \langle \rho(f^*)\tilde{g},\eta\rangle	\\
		=  \langle (f^**g)\tilde{},\eta\rangle\\
		= \langle \tilde{g}, \tilde{f}\rangle.	
		$$
		It follows that, $\rho(f)\eta=\tilde{f}$, for  $f\in \ell^{1}(\mathcal{B})$.
		Therefore,
		\begin{eqnarray*}
			\pi_{\eta,\eta,r}(b)&=&\langle\rho(J_r(b))\eta,\eta \rangle
			=\langle \widetilde{J_r(b)},\eta\rangle\\
			&=& \varphi(J_r(b))
			=\sum_{s \in G}\varphi_s(J_r(b)(s))=\varphi_r(b),	
		\end{eqnarray*}
		for $r\in G, b\in B_r.$ Thus, $\varphi=\pi_{\eta,\eta}\in P(\mathcal{B})$.	
		
		$(ii)$ To each  $\varphi=\pi_{\xi,\eta}\in B(\mathcal{B})$ one could associate $\tilde\varphi\in C^*(\mathcal{B})^*$, defined by $\varphi(u)=\langle\tilde{\pi}(u)\xi,\eta\rangle$, for
		$u\in C^*(\mathcal{B})$. Using Jordan decomposition
		$
		\varphi=(\varphi_1-\varphi_2)+i(\varphi_3-\varphi_4)
		$
		into positive linear functionals, the assertion now follows by part $(i)$.	
		
		$(iii)$
		Define $T:B(\mathcal{B})\rightarrow C^{*}(\mathcal{B})^*$ by $\varphi\mapsto T_\varphi$, where
		$T_\varphi:C^{*}(\mathcal{B})\rightarrow\mathbb{C}$	is defined
		by $T_\varphi(f)=\sum_{s\in G}\varphi_s(f_s)$, for each $f\in \ell^1(\mathcal{B})$.
		By part $(ii)$,  $\|T_\varphi\|=\|\varphi\|_{B(\mathcal{B})}$, as required.
		
		Next, for $u\in C^{*}(\mathcal{B})^*$ with Jordan decomposition  	$$
		u=(u_1-u_2)+i(u_3-u_4),
		$$
		for each  positive  functional $u_i$, by part $(i)$, there is a
		representation $\pi^i\in\Sigma_c(\mathcal{B})$ on a Hilbert space $H_{i}$
		and $\xi_i\in H_{i}$ such that $u_i=\pi^i_{\xi_i,\xi_i}$.
		Define
		$
		\rho:=\{\rho_s\}_{s\in G}
		$
		on the direct sum $H_0$ of the Hilbert spaces $H_i$
		by
		$$
		\langle\rho_s(b)(\xi_i),
		(\eta_i)\rangle=\sum_{i=1}^{4}\langle\pi^i_s(b)\xi_i,\eta_i\rangle.
		$$
		Then, for
		$\xi_0=(\xi_1,-\xi_2,i\xi_3,-i\xi_4)$ and $\eta_0=(\xi_1,\xi_2,\xi_3,\xi_4)$ in $H_0$,
		$$
		\langle\rho_s(b)\xi_0,\eta_0 \rangle=\langle\pi^1_s(b)\xi_1,\xi_1\rangle-
		\langle\pi^2_s(b)\xi_2,\xi_2\rangle+i\langle\pi^3_s(b)\xi_3,\xi_3\rangle-i\langle\pi^4_s(b)\xi_4,\xi_4\rangle
		=u(b),	
		$$
		that is, $u=\rho_{\xi_0,\eta_0}$ and  $T_\rho=u$.
		
		$(iv)$
		Let $\varphi=R_h\psi$ be the polar decomposition of $\varphi$, where $\psi\in B(\mathcal{B})$ is positive,
		$\|\psi\|=\|\varphi\|$,
		$h\in B(\mathcal{B})^*=C^*(\mathcal{B})^{**}$ with $\|h\|\leq 1$,
		and $R_h\psi(g)=\psi(gh)$,	for  $g\in B(\mathcal{B})^*$. By Lemma \ref{positivity},
		there exists a representation $\pi\in\Sigma_{c}(\mathcal{B})$ and a vector $\eta\in H_\pi$ such that
		$\psi=\pi_{\eta,\eta}$.
		Let ${\omega}:C^*(\mathcal{B})\rightarrow C^*(\mathcal{B})^{**}$ be the
		universal representation of $C^*(\mathcal{B})$. Then, $\tilde{\pi}(g)=\tilde{\pi}^{**}(\omega(g))$,
		for  $g\in C^*(\mathcal{B})$, where $\tilde{\pi}^{**}$ is the second transpose of
		$\tilde{\pi}$ from $C^*(\mathcal{B})^{**}$ into $L(H_\pi)$.  	
		Put $\xi:=\tilde{\pi}^{**}(h)\eta$. By Lemma \ref{eta2},
		$$
		\|\tilde{\pi}^{**}(h)\eta\|^2=|\langle \tilde{\pi}^{**}(h^**h)\eta,\eta \rangle|=
		|\langle h^**h,\psi\rangle|\leq \|\psi\|=\|\varphi\|=\|\eta\|^2.
		$$
		Therefore, $\|\xi\|\leq \|\eta\|$. Next, for every $g\in C^*(\mathcal{B})$,
		\begin{eqnarray*}
			\langle\varphi, g\rangle&=&\langle R_h\psi,g\rangle
			=\langle	\psi,\omega(g)*h\rangle\\
			&=&\langle\tilde{\pi}^{**}(\omega(g))\tilde{\pi}^{**}(h)\eta,\eta \rangle	
			=\langle \tilde{\pi}(g)\xi,\eta \rangle,
		\end{eqnarray*}	
		that is, $\varphi=\pi_{\xi,\eta}$, thus,
		$$
		\|\varphi\|\leq\|\xi\|\|\eta\|\leq\|\eta\|^2=\|\psi\|=\|\varphi\|.
		$$	
		
		$(v)$ This follows immediately from part $(iv)$. 	  	
	\end{proof}
	
	\begin{example} \label{fs} Certain basic examples of the Fourier-Stieltjes space are in order.
		
		\vspace{.2cm}
		$(i)$ Let $G$ be a  finite abelian group and  $(A,G,\tau)$	be a $C^*$-dynamical system. For each  character $\chi\in\widehat{G}$, put
		$B_\chi=\{a\in A:\tau_s(a)=\chi(s)a~ (s\in G)\}$. This forms a Fell
		bundle $\mathcal{B}=\{B_\chi\}_{\chi\in\widehat{G}}$  over $\widehat{G}$ and $A$ can be identified with $C^*(\mathcal{B})$; for more details see \cite[P. 1044]{FeDo88}.
		This is known to be saturated when $A$ is of compact type (i.e., a direct sum of elementary $C^*$-algebras; cf., \cite[VI.23.3]{FeDo88}) and $G$ acts on $A$ by inner automorphisms \cite[X.Exercise 23]{FeDo88}. In this case,
		the Fourier-Stieltjes space $B(\mathcal{B})$ is isomorphic to $A^*$. As far as we know, a necessary and sufficient condition for the bundle in this example to be saturated is not known yet \cite[page 1045]{FeDo88}.
		
		\vspace{.2cm}
		$(ii)$
		Let  $\mathcal{B}$ be the Fell bundle of a $C^*$-dynamical system  $(A,G,\gamma)$.
		To every functional $\varphi:A\rtimes_\gamma G\rightarrow\mathbb{C}$
		associate a function $\Phi:G\rightarrow A^*$ by $\Phi(t)(a)=\varphi(au_t)$, for  $t\in G$ and $a\in A$; where
		$u:G\rightarrow M(A\rtimes_\gamma G)$ pairs with  a representation
		$\pi:A\rightarrow M(A\rtimes_\gamma G)$ in the universal covariant representation  $(\pi,u)$  of $(A,G,\gamma)$.
		The set of functions $\Phi:G\rightarrow A^*$ is denoted by $B(A\rtimes_\gamma G)$ and there is a
		bijection from $(A\rtimes_\gamma G)^*$ onto $B(A\rtimes_\gamma G)$, as observed first by Gert Pedersen in \cite{pedersen}.
		Moreover, if $\mathcal{B}=\{B_s\}_{s\in G}$ is the Fell bundle of $(A,G,\gamma)$, then it is known that
		$C^*(\mathcal{B})\simeq A\rtimes_\gamma G$. Therefore, $B(\mathcal{B})\simeq (A\rtimes_\gamma G)^*$,
		since  $\mathcal B$ is known to be automatically saturated in this case \cite[VIII.4.3]{FeDo88}.
		
		For this example, the Fourier-Stieltjes space
		$B(\mathcal{B})$ is isomorphic to the space $B(A\rtimes_\gamma G)$, introduced by Pedersen \cite[7.6.6]{pedersen}. If $\varphi$ is a
		positive functional on $A\rtimes_\gamma G$,  then the associated function $\Phi$ is a  positive definite function. The above isomorphism maps $P(\mathcal{B})$ on the
		set $B_+(A\rtimes_\gamma G)$ of  positive definite functions (cf., \cite[7.6.8]{pedersen}).
		
		$(iii)$ For the Fell bundle $\mathcal A$ associated as in Example \ref{bundle}($iii$) to a coaction  $\delta: A \rightarrow$ $A \otimes C^{*}(G),$   we have $A\cong C^*(\mathcal A)$. Since this bundle is clearly saturated, we get $B(\mathcal A)\simeq A^*$.

		$(iv)$ If $G$ is abelian, then there exists a natural action
		$\widehat{\gamma}$ of the dual group $\widehat{G}$ of $G$ on $B:=A \rtimes_\gamma  G$ given for $f \in C_{c}(G, A)$ by
		$$
		\widehat{\gamma}_{\chi}(f)(s)=\chi(s) f(s).
		$$
		The  Takesaki-Takai duality theorem asserts that
		$(A \rtimes_{\gamma} G) \rtimes_{\hat{\gamma}} \widehat{G}\simeq A \otimes \mathcal{K}\left(\ell^{2}(G)\right)$, canonically, that is, the double crossed product is stably isomorphic to $A$.
		One can define the Fell bundle of $(A \rtimes_\gamma  G,\widehat{G},\hat{\gamma})$ by
		$\widehat{\mathcal{B}}=\{\widehat{B}_\chi\}_{\chi\in\widehat{G}}$, which gives
		$$C^*(\widehat{\mathcal{B}})\simeq (A \rtimes_{\gamma} G) \rtimes_{\hat{\gamma}} \widehat{G}\simeq
		C^*(\mathcal{B})\rtimes_{\hat{\gamma}}\widehat{G}. $$
		Therefore,    $B(\widehat{\mathcal{B}})\simeq ( C^*(\mathcal{B})\rtimes_{\hat{\gamma}}\widehat{G})^*$
		if   $\widehat{\mathcal{B}}$ is saturated (this in particular holds if $A \rtimes_\gamma  G$ is of compact type and $\hat\gamma$ is inner; see part ($i$) above).

	\end{example}

	Let $\lambda:\mathcal{B}\rightarrow\mathcal{L}\left(\ell^{2}(\mathcal{B})\right)$ be the  regular
	representation of a Fell bundle  	$\mathcal{B}=\{B_s\}_{s\in G}$ on the
	Hilbert  $B_e$-module
	$\ell^2(\mathcal{B})$. We may then consider the map
	$$
	\lambda_{\xi, \eta,s}:B_s\rightarrow B_e;\quad  \lambda_{\xi, \eta,s}(b)=\langle\lambda_s(b)\xi,\eta\rangle,
	$$
	for  $\xi,\eta\in\ell^2(\mathcal{B})$, $s\in G$ and $b\in B_s$.

	\begin{definition}	
		The \textit{Fourier} space $A(\mathcal{B})$	of a Fell bundle	$\mathcal{B}$ is defined by
		$$A(\mathcal{B})= \left\{\lambda_{\xi,\eta}:\xi, \eta \in \ell^{2}(\mathcal{B}) \right\}.
		$$
		Put
		$$
		\|\lambda_{\xi,\eta}\|=\sup\{\|\langle\lambda(f)\xi,\eta\rangle\|: f\in\ell^1(\mathcal{B}),~\|f\|_*\leq 1\},
		$$
		where, $\langle\lambda(f)\xi,\eta\rangle:=\sum_{r\in G}\langle\lambda_r(f_r)\xi,\eta\rangle$.
	\end{definition}
	First let us observe that $A(\mathcal{B})$
	is a two-sided $A(G)$-module:
	given $\lambda_{\xi, \eta}\in A(\mathcal{B})$	and $\lambda^G_{\alpha,\beta}\in A(G)$, for every
	$s\in G$ and $b\in B_s$,
	$$
	(\lambda_{\xi,\eta}\cdot\lambda^G_{\alpha,\beta})_s(b)=(\lambda^G_{\alpha,\beta}\cdot\lambda_{\xi,\eta})_s(b)=
	\langle\lambda_s(b)\xi,\eta \rangle\langle\lambda_s^G\alpha,\beta\rangle\in B_e.
	$$

	\begin{example}	\label{f} Fourier space could be calculated for typical cases of Fell bundles.
		
		\vspace{.2cm}
		$(i)$ Let $\mathcal{B}$ be the Fell bundle of a $C^*$-dynamical system   $(A,G,\gamma)$.
		By Example \ref{bundle}$(ii)$, we have the identification
		$$\ell^2(\mathcal{B})=\{(a_s)\in\Pi_{s\in G} B_s : \|\sum_{s \in G}\gamma_{s^{-1}}(a_s^*a_s)\|<\infty\}.$$
		In this case, $\lambda_{\xi, \eta,s}(b)=(e,\sum_{t \in G}a_{s^{-1}t}^*b^*b_t)\in B_e$. Back to Example \ref{fs}($ii$), the isomorphism sending $B(\mathcal B)$ onto $B(A\rtimes_\gamma G)$ sends  $A(\mathcal B)$ onto $A(A\rtimes_\gamma G)$, defined by Pedersen \cite[page 262]{pedersen}. An element in $B_+(A\rtimes_\gamma G)$ belongs to $A_+(A\rtimes_\gamma G)$ \cite[7.7.6]{pedersen} and $A(A\rtimes_\gamma G)$ is nothing but the span closure of elements of ``finite support'' in $B_+(A\rtimes_\gamma G)$ (where finite support here simply means that the corresponding map on $G$ is of finite support). It now follows  that $A(\mathcal{B})$ coincides with the space $A(A\rtimes_\gamma G)$ considered by Pedersen, when $\mathcal{B}$ is the Fell bundle of system   $(A,G,\gamma)$. In particular, it follows from (the proof of) \cite[7.7.7]{pedersen} that for Fell bundles coming from actions, $A(\mathcal{B})=B(\mathcal{B})$ when $G$ is discrete amenable (we make these vague ideas into rigorous proofs in the next section).

		$(ii)$ Let $\mathcal{B}=\{A_s\}_{s\in G}$ be the Fell bundle whicof a coaction
		$\delta: A \rightarrow$ $A \otimes C^{*}(G)$ as in Example \ref{bundle}$(iii)$. Then
		$$\ell^2(\mathcal{B})=\{(a_s)\in\Pi_{s\in G} A_s : \|\sum_{s \in G}a_s^*a_s\|<\infty\}.$$
		In this case, $\lambda_{\xi, \eta,s}(b)=\sum_{t \in G}a_{s^{-1}t}^*b^*b_t$,
		for  $\xi=(a_t)_{t\in G},\eta=(b_t)_{t\in G}\in\ell^2(\mathcal{B})$ and $b\in A_s$. Back to Example \ref{fs}($iii$), we say that a positive functional $\phi\in A^*_+$ is of ``compact support'' if $\phi(\sum_{s\in G}a_s^*a_s)<\infty$. One could now identify the Fourier space $A(\mathcal A)$ of the Fell bundle $\mathcal A$ of a coaction with the span closure of positive functionals.
		
	\end{example}

	Let  $c_0(G,\mathcal{L}(\mathcal{B},B_e))$ be the Banach space of all maps $\varphi$
	from $G$ into disjoint union of spaces $\mathcal{L}({B_s},B_e)$, vanishing at infinity with
	$\varphi(s)\in\mathcal{L}(B_s,B_e)$, for every $s\in G$, with norm  $\|\varphi\|_{\infty}=\sup_{s\in G}\|\varphi(s)\|$.
	We denote  by $c_c(G,\mathcal{L}(\mathcal{B},B_e))$ the
	set of all $\varphi\in c_0(G,\mathcal{L}(\mathcal{B},B_e))$ with finite support.
	
	It is known that for a (locally compact) group $G$, the Fourier algebra $A(G)$ is a subspace of $C_0(G)$. This also holds for Fell bundles in  the
	following sense.
	
	\begin{proposition}
		$A(\mathcal{B})$ is a subspace of $c_0(G,\mathcal{L}(\mathcal{B},B_e))$.
	\end{proposition}
	\begin{proof}
		Given $\varphi,\psi\in C_c(\mathcal{B})$ with
		$K:=\operatorname{supp}(\varphi)$ and $L:=\operatorname{supp}(\psi)$, $KL^{-1}$ is compact, and given $s\notin KL^{-1}$, it is easy to see
		that $\lambda_{\varphi,\psi,s}=0$. Therefore, $\lambda_{\varphi,\psi}\in
		c_c(G,\mathcal{L}(\mathcal{B},B_e))$. If $\xi,\eta\in\ell^2(\mathcal{B})$,  there exist sequences
		$(\varphi_n)$ and $(\psi_n)$ in $c_c(\mathcal{B})$ such that
		$\|\varphi_n-\xi\|_2\rightarrow 0$ and $\|\psi_n-\eta\|_2\rightarrow 0$. We have,
		\begin{eqnarray*}
			\|\lambda_{\varphi_n,\psi_n}-\lambda_{\xi,\eta}\|_\infty&=&
			\sup_{s\in G}\|\lambda_{\varphi_n,\psi_n,s}-\lambda_{\psi,\eta,s}\|\\
			&\leq&\|\varphi_n\|_2\|\psi_n-\eta\|_2+\|\varphi_n-\xi\|_2\|\eta\|_2\rightarrow 0.    	 	
		\end{eqnarray*}
		Thus, $\lambda_{\xi, \eta}\in c_0(G,\mathcal{L}(\mathcal{B},B_e))$.
	\end{proof}

	\section{multiplier algebra and bounded approximate identities} \label{4}

	In this section we find sufficient  conditions on Fell bundles for the Fourier and Fourier-Stieltjes spaces to be Banach algebras.

	Let $\mathcal{A}=\{A_s\}_{s\in G}$  and $\mathcal{B}=\{B_t\}_{t\in H}$ be Fell bundles over
	discrete groups $G$ and $H$. For  $s\in G$ and $t\in H$, consider the algebraic tensor product
	$A_s\odot B_t$ and put $\mathcal{A}\odot\mathcal{B}=\{A_s\odot B_t\}_{(s,t)\in G\times H}$. This is called the {\it algebraic tensor product} of $\mathcal{A}$ and $\mathcal{B}$. Let $p$ be a $C^*$-norm on
	$A_e\odot B_e$ satisfying $p(x^*x)=p(xx^*)$, for all $x\in \mathcal{A}\odot\mathcal{B}$. Put
	$\|x\|_{p}=\sqrt{p(x^*x)}$ and let $A_s\otimes_p B_t$ be the corresponding completion of $A_s\odot B_t$.    Then, $\mathcal{A}\otimes_p\mathcal{B}:=\{A_s\otimes_p B_t\}_{(s,t)\in G\times H}$ is a
	Fell bundle over  $G\times H$. We refer the reader  for more details to \cite{Ab97}.
	
	The notion of {\it Hopf coalgebra} is introduced by Turaevin in \cite{Tura2000} and  studied
	by Virelizier in \cite{Vi02}. Roughly speaking, a Hopf $G$-coalgebra is a family $D=\{D_s\}_{s\in G}$ of algebras
	endowed with a co-multiplication $\Delta=\{\Delta_{s,t} : D_{st}\rightarrow D_s\otimes D_t\}_{s,t\in G}$, a counit $\epsilon:D_e\rightarrow\mathbb{C}$
	and an antipode $S=\{S_t:H_t\rightarrow H_{t^{-1}}\}$, satisfying natural compatibility conditions.

	\begin{proposition}
		Let $\mathcal{B}=\{B_s\}_{s\in G}$ be a Fell bundle over $G$ endowed with  a family
		$\Delta=\{\Delta_{s,t} : B_{st}\rightarrow B_s\otimes_p B_t\}_{s,t\in G}$  of bounded linear maps
		satisfying
		$$
		(\Delta_{r,s}\otimes\mathrm{id}_{B_t})\Delta_{rs,t}=(\mathrm{id}_{B_r}\otimes\Delta_{s,t})\Delta_{r,st},
		$$
		for  $r,s,t\in G$. Then $B(\mathcal{B})$ and $A(\mathcal{B})$ are Banach algebras.
	\end{proposition}
	\begin{proof}
		Given  representations  $\pi$ and $\rho$ of the bundle $\mathcal{B}=\{B_s\}_{s\in G}$,  we define
		$(\pi\otimes\rho)_{(s,t)}: B_s\otimes_p B_t\rightarrow\mathcal{L}(H_\pi\otimes H_\rho)$ by
		$$
		(\pi\otimes\rho)_{(s,t)}(b\otimes c):=(\pi_s\otimes\rho_t)\Delta_{s,t}(bc),\quad (s,t\in G, b\in B_s, c\in B_t).
		$$
		Then, $\pi\otimes\rho$ is a representation for Fell bundle
		$\mathcal{B}\otimes_p\mathcal{B}=\{B_s{\otimes}_p B_t\}_{(s,t)\in G\times G}$ over $G\times G$.
		For $\xi,\eta\in H_\pi$, $\alpha,\beta\in H_\rho$, and  $s\in G$, and  a fixed element $b^s_0\in B_s$, put
		$$
		(\pi_{\xi,\eta}\rho_{\alpha,\beta})_s(b):=(\pi\otimes\rho)_{\xi\otimes\alpha,\eta\otimes\beta,(s,s)}(b\otimes b^s_0),
		$$ and observe that $$ (\pi_{\xi,\eta}\rho_{\alpha,\beta})_s(b)=\langle(\pi_s\otimes\rho_s)\Delta_{s,s}(bb^s_0)
		(\xi\otimes\alpha),\eta\otimes\beta\rangle.$$ In particular,  $\pi_{\xi,\eta}\rho_{\alpha,\beta}\in B(\mathcal{B})$
		and it is not difficult to see that $\|\pi_{\xi,\eta}\rho_{\alpha,\beta}\|\leq\|\pi_{\xi,\eta}\|\|\rho_{\alpha,\beta}\|$.
		Thus, $B(\mathcal{B})$ is a Banach algebra.

		Next, we may view $\ell^2(\mathcal{B})\otimes\ell^2(\mathcal{B})$ as a right Hilbert $B_e$-module via
		$(\xi_1\otimes\xi_2)\cdot a=\xi_1\otimes\xi_2\cdot a $ and   $B_e$-valued inner product
		$$
		\langle\xi_1\otimes\xi_2,\xi_3\otimes\xi_4\rangle=\langle\xi_2,\lambda_e(\langle\xi_1,\xi_3\rangle)\xi_4\rangle.
		$$
		Define
		$(\lambda\otimes\lambda)_{s,t}:B_s\otimes_{\mathrm{min}} B_t\rightarrow\mathcal{L}(\ell^2(\mathcal{B})\otimes\ell^2(\mathcal{B}))$
		by $$(\lambda\otimes\lambda)_{s,t}(b\otimes c)=(\lambda_s\otimes\lambda_t)\Delta_{s,t}(bc),$$
		and observe that,
		\begin{eqnarray*}
			(\lambda_{\xi_1,\xi_2}\lambda_{\xi_3,\xi_4})_s(b)&:=&
			(\lambda\otimes\lambda)_{\xi_1\otimes\xi_3,\xi_2\otimes\xi_4,(s,s)}(b\otimes b_0^s)\\
			&=&\langle(\lambda_s\otimes\lambda_s)\Delta_{s,s}(bb^s_0)
			(\xi_1\otimes\xi_3),\xi_2\otimes\xi_4\rangle\in B_e.
		\end{eqnarray*}
		A similar norm inequality holds in this case as well, and $A(\mathcal{B})$ is also a Banach algebra. 	
	\end{proof}

	\begin{example} We provide concrete examples of cases in which Fourier-Stieltjes space is a Banach algebra.
		
		\vspace{.2cm}
		$(i)$ Let $(A,G,\gamma)$ be a $C^*$-dynamical system such that $A$ is a Hopf $C^*$-algebra; i.e., there is a
		comultiplication $\Delta:A\rightarrow M(A\otimes_{\mathrm{min}}A)$. It is known that $A^*$ is a Banach algebra under
		$fg=(f\otimes g)\circ\Delta$, $f,g\in A^*$. Suppose that $\mathcal{B}=\{A_s\}_{s\in G}$ is the Fell bundle
		of $(A,G,\gamma)$.
		Let $\pi$ and $\rho$ be  representations of $\mathcal{A}$ on Hilbert spaces $H_\pi$ and $H_\rho$.
		Then each $A_s$  is  isomorphic as a Banach space to $A$  and for  $\xi,\eta\in H_\pi$, $\alpha,\beta\in H_\rho$, and $s\in G$, one may
		consider $\pi_{\xi,\eta,s}$ and $\rho_{\alpha,\beta,s}$ as  bounded linear functionals on $A$. This allows ud to consider
		$\pi_{\xi,\eta,s}\otimes\rho_{\alpha,\beta,s}:M(A\otimes_{\mathrm{min}}A)\rightarrow\mathbb{C}$, defined by
		$$
		\pi_{\xi,\eta,s}\otimes\rho_{\alpha,\beta,s}(a\otimes b)=\langle\pi_s((s,a))\xi,\eta\rangle\langle\rho_s((s,b))\alpha,\beta\rangle.
		$$
		Thus, $\pi\otimes\rho$ is a representation of the Fell bundle $\mathcal{A}\otimes_{\mathrm{min}}\mathcal{A}=\{A_s\otimes_{\mathrm{min}}A_t\}_{(s,t)\in G\times G}$ on
		Hilbert space $H_\pi\otimes H_\rho$, with
		$$
		(\pi\otimes\rho)_{(s,t)}(a\otimes b)=\pi_s((s,a))\otimes\rho_t((t,b)).
		$$
		
		Finally, given $\pi_{\xi,\eta}, \rho_{\alpha,\beta}\in B(\mathcal{A}) $,
		$$
		(\pi_{\xi,\eta} \rho_{\alpha,\beta})_s((s,a))=\langle(\pi_{\xi,\eta,s}\otimes\rho_{\alpha,\beta,s}),\Delta(a)\rangle,
		$$
		showing that  the Fourier-Stieltjes space $B(\mathcal{A})$ is closed under multiplication. The other conditions are now straightforward.
		
		\vspace{.2cm}
		$(ii)$ Let $(M,\Delta)$ be a Hopf-von Neumann algebra and $(M,G,\gamma)$ be a $W^*$-dynamical system.
		Let  $\mathcal{M}=\{M_s\}_{s\in G}$ be the corresponding Fell bundle.
		For every $s\in G$, define
		$f_s:M\rightarrow M_s$ by $f_s(a)=(s,a)$. It is straightforward to see that
		$ f_r\otimes f_s:M\bar{\otimes}M\rightarrow M_r{\otimes}_p M_s$ is a bounded linear map,
		where $(f_r\otimes f_s)(a\otimes b)=f_r(a)\otimes f_s(b)$, for  $r,s\in G$ and $a,b\in M$, where the domain is the von Neumann algebra tensor product.
		The coproduct family
		$$
		\Delta_{r,s}:M_{rs}\rightarrow M_r{\otimes}_p M_s,\quad\Delta_{r,s}((rs,a))=(f_r\otimes f_s)(\Delta(a))
		$$
		satisfies $ (\Delta_{r,s}\otimes\mathrm{id}_{B_t})\Delta_{rs,t}=(\mathrm{id}_{B_r}\otimes\Delta_{s,t})\Delta_{r,st}$.	
		In particular,  the Fourier-Stieltjes space $B(\mathcal{B})$ and Fourier space $A(\mathcal{B})$ are Banach algebras.
		In fact, given $\pi_{\xi,\eta},\rho_{\alpha,\beta}\in B(\mathcal{B})$,
		\begin{eqnarray*}
			(\pi_{\xi,\eta}\rho_{\alpha,\beta})_s(s,b)&=&\langle(\pi_s\otimes\rho_s)\Delta_{s,s}(s^2,b)\xi\otimes\alpha,
			\eta\otimes\beta\rangle\\
			&=&\langle(\pi_s\otimes\rho_s)(f_s\otimes f_s)(\Delta(b))\xi\otimes\alpha,
			\eta\otimes\beta\rangle.	 	
		\end{eqnarray*}

		\vspace{.2cm}
		$(iii)$ For the trivial Fell bundle $\mathcal{B}=G\times\mathbb{C}=\{B_s\}_{s\in G}$,
		let us observe  that $A(\mathcal{B})$ is canonically isomorphic to the classical Fourier algebra
		$A(G)$.
		Define
		$$
		\Delta_{s,t}:B_{st}\rightarrow B_s{\otimes}_p B_t,\quad (st,z)\mapsto (s,z){\otimes} (t,1).
		$$
		This a compatible coproduct family. Define,
		$\theta:\mathcal{B}\rightarrow\mathbb{C}$ by $\theta(b)=z_b$, for  $s\in G$ and $b=(s,z_b)\in B_s$.
		Then, $\theta(b^*)=\overline{z_b}$ and
		$$
		\langle\lambda_s(b)\xi,\eta\rangle=\left(e,\overline{z_b}\langle\lambda^G_s\theta\xi,
		\theta\eta\rangle_{\ell^2(G)}\right),
		$$	
		where $\lambda^G$ is the left regular representation of $G$.
		Define $T_b:A(\mathcal{B})\rightarrow A(G)$ by
		$$
		T_b(\lambda_{\xi,\eta})(s)=\frac{1}{\overline{z_b}}\theta\left(\langle\lambda_s(b)\xi,\eta\rangle\right)=
		\langle\lambda^G_s{\theta\xi,\theta\eta}\rangle_{\ell^2(G)},
		$$
		where $b=(s,z_b)$ is a nonzero element of $B_s$.
		Consider the map $T_b:A(\mathcal{B})\rightarrow A(G)$. It is easy
		to see that $T_b$ is a bounded linear map. Moreover,
		\begin{eqnarray*}
			T_b(\lambda_{\xi,\eta}\lambda_{\alpha,\beta})(s) &=&T_b((\lambda\otimes\lambda)_{\xi\otimes\alpha,\eta\otimes\beta})(s)\\
			&=&\frac{1}{\overline{z_b}}\theta(\langle(\lambda_s\otimes\lambda_s)\Delta_{s,s}(b(s,1))(\xi\otimes\alpha),
			\eta\otimes\beta\rangle) \\
			&=&	\frac{1}{\overline{z_b}}\theta(\langle(\lambda_s\otimes\lambda_s)(b\otimes(s,1))(\xi\otimes\alpha),
			\eta\otimes\beta\rangle)\\
			&=& \frac{1}{\overline{z_b}}\theta\left(\langle\lambda_s(b)\xi,\eta\rangle\otimes
			\langle\lambda_s((s,1))\alpha,\beta\rangle\right)\\
			&=&\frac{1}{\overline{z_b}}\theta\left((e,\overline{z_b}\langle\lambda^G_s{\theta\xi,\theta\eta}\rangle_{\ell^2(G)})
			\otimes(e,\langle\lambda^G_s{\theta\alpha,\theta\beta}\rangle_{\ell^2(G)}) \right)\\
			&=&\langle\lambda^G_s{\theta\xi,\theta\eta}\rangle_{\ell^2(G)}
			\langle\lambda^G_s{\theta\alpha,\theta\beta}\rangle_{\ell^2(G)}\\
			&=&T_b(\lambda_{\xi,\eta})(s)T_b(\lambda_{\alpha,\beta})(s).
		\end{eqnarray*}
		Thus, $T_b$ is homomorphism. In fact, $T_b(\lambda_{\xi,\eta})=\lambda^G_{\theta\xi,\theta\eta}$, and
		since $\|\lambda_{\xi,\eta}\|=\|\lambda^G_{\theta\xi,\theta\eta}\|$, it follows that  	
		$T_b$ is an isometric isomorphism.
	\end{example}

	Back to general Fell bundles, since $B_sB_e\subseteq B_s$, it follows that
	$\ell^2(\mathcal{B})$ is a Hilbert $B_e$-module if we put $(\xi\cdot a)(s)=\xi(s)a$, for
	$\xi\in\ell^2(\mathcal{B})$ and $a\in B_e$. Then
	for any Hilbert space $H$,
	$\ell^2(\mathcal{B})\otimes H$ is a Hilbert
	$B_e$-module via $(\xi\otimes h)\cdot a=\xi\cdot a\otimes h$ and
	$$
	\langle\xi\otimes h,\eta\otimes k\rangle=\langle h,k\rangle\langle\xi,\eta\rangle\in B_e,
	\quad (\xi,\eta\in\ell^2(\mathcal{B})\;\mathrm{and}\; h,k\in H).
	$$

	For a representation $\pi:\mathcal{B}\rightarrow\mathcal{L}(H_\pi)$,
	consider
	$\lambda\otimes\pi:\mathcal{B}\rightarrow\mathcal{L}(\ell^2(\mathcal{B})\otimes H_\pi)$, defined by
	$(\lambda\otimes\pi)_s(b):=(\lambda_s\otimes\pi_s)\Delta_{s,s}(bb^s_0)$, for
	$s\in G$ and $b\in B_s$.

	For a Banach algebra $A$, let $M(A)$ be the Banach algebra of double centralisers, that is,
	pairs $(L, R)$ of linear maps on $A$ with
	$$
	\left\{\begin{array}{c}
		L(a b)=L(a) b, \quad R(a b)=a R(b) \\
		a L(b)=R(a) b
	\end{array} \quad(a, b \in \mathcal{A})\right..
	$$
	We denote by $M(A(\mathcal{B}))$ the multiplier algebra of $A(\mathcal{B})$.
	In the group case, $A(G)$ has a bounded approximate identity if and only if $M(A(G))=B(G)$. We show that a similar statement holds for Fell bundles.
	
	\begin{lemma}\label{representation}
		Let $\mathcal{B}$ be a Fell bundle over a discrete group $G$.
		If $\pi$ is a  representation of $\mathcal{B}$, then  $\lambda_{\xi,\eta}\pi_{\alpha,\beta}\in A(\mathcal{B})$.
	\end{lemma}
	\begin{proof}
		Let $\lambda_{\xi,\eta}\in A(\mathcal{B})$ and $\pi_{\alpha,\beta}\in B(\mathcal{B})$. Then
		\begin{eqnarray*}
			(\lambda_{\xi,\eta}\pi_{\alpha,\beta})_s(b)&=&\langle(\lambda_{s} \otimes \pi_{s})\Delta_{s,s}(bb^s_0)(\xi\otimes\alpha),
			\eta\otimes\beta \rangle	\\
			&=&\langle(\lambda_s\otimes \pi_s)\sum_{i=1}^{\infty}(b_i\otimes d_i)(\xi\otimes\alpha),\eta\otimes\beta\rangle\\
			&=&\sum_{i=1}^{\infty}\langle\lambda_s(b_i)\xi\otimes\pi_s(d_i)\alpha,\eta\otimes\beta\rangle\\
			&=&\sum_{i=1}^{\infty}\langle\lambda_s(b_i)\xi,\eta\rangle\langle\pi_s(d_i)\alpha,\beta\rangle\in B_e.
		\end{eqnarray*}
		Thus,  $\lambda_{\xi,\eta}\pi_{\alpha,\beta}\in A(\mathcal{B})$.
	\end{proof}
	
	\begin{theorem}
		Let $\mathcal{B}$ be a Fell bundle over a discrete group $G$. If $A(\mathcal{B})$ is a
		Banach algebra with  a
		bounded approximate identity, then $M(A(\mathcal{B}))= B(\mathcal{B})$. 	
	\end{theorem}
	\begin{proof}
		Let $(e_{i})_i=(\lambda_{\xi_i,\eta_i})_i$ be a bounded approximate identity for $A(\mathcal{B})$ with $\| e_{i}\|= 1$.
		By Lemma \ref{representation},  for any $w\in B(\mathcal{B})$,
		the bounded linear maps $L_w:A(\mathcal{B})\rightarrow A(\mathcal{B})$ and
		$R_w:A(\mathcal{B})\rightarrow A(\mathcal{B})$, given by
		$L_w(v)=wv$ and $R_w(v)=vw$, are well defined.
		
		Let
		$(L,R)\in M(A(\mathcal{B}))$.
		Each element $R(e_{i})$ may be regarded as bounded linear map from $C^{*}(\mathcal{B})$ into $B_e$ via
		$$
		R(e_i)(f)=\sum_{s\in G}\langle\lambda_s(f_s)\xi_i,\eta_i \rangle , \quad (f\in \ell^1(\mathcal{B})).
		$$
		Since, $\sup_{i\in I}\|R(e_i)(f)\|<\infty$, by the uniform boundedness principle, there exists
		$w\in\mathcal{L}(C^{*}(\mathcal{B}),B_e)$ such that $w(f)=\lim_i R(e_i)(f)$, for  $f\in\ell^1(\mathcal{B})$.
		
		Fix $u \in A(\mathcal{B}))$. Then,
		$$
		uw(f)=\lim_i(uR(e_i))(f)=\lim_i R(ue_i)(f)=R(u)(f),\quad(f\in\ell^1(\mathcal{B})).
		$$
		Therefore, $R=R_w$, for some $w\in B(\mathcal{B})$, where
		$R_w(v)=vw$.
		Similarly,
		$L(v)=L_w(v)$, where
		$L_w(v)=wv$. Thus, $M(A(\mathcal{B}))= B(\mathcal{B})$, as required.
	\end{proof}
	
	Let $(A,G,\gamma)$ be a $C^*$-dynamical system such that $A$ is unital with unit element $1$ and let
	$\mathcal{B}=\{B_s\}_{s\in G}$ be the corresponding Fell bundle. For any subset $V\subseteq G$, define
	$1_V:G\rightarrow\mathcal{B}$ by $1_V(s)=(s,1)$, for $s\in V$, and $1_V(s)=(s,0)$, otherwise.
	
	The next lemma extends Eymard's trick \cite{Eymard64}.
	
	\begin{lemma}\label{eym-trick}
		Let $\mathcal{B}=\{B_s\}_{s\in G}$ be the Fell bundle of a C*-dynamical system $(A,G,\gamma)$,
		$U$ be a subset of  $G$ and  $K\subseteq U$ be a compact set.	
		There exists $u\in A(\mathcal{B})$ such that $u_s(s,a)=(e,a^*)$, for $s\in K$ and
		$u_s(s,a)=(s,0)$, otherwise.  	
	\end{lemma}
	\begin{proof}
		Let $V$ be a finite neighborhood  	of $e\in G$. Define $u=\frac{1}{|V|}\lambda_{1_V,1_{KV}}$, where
		$|V|$ is the cardinal of $V$. Then
		\begin{eqnarray*}
			u_s(s,a)&=&\frac{1}{|V|}\langle\lambda_s(s,a)1_V,1_{KV}\rangle\\
			&=&\frac{1}{|V|}\sum_{r\in G}(\lambda_s(s,a)1_V(r))^*1_{KV}(r)\\
			&=&\frac{1}{|V|}\sum_{r\in G}(1_V(s^{-1}r))^*(s^{-1},a^*)1_{KV}(r)\\
			&=&\frac{1}{|V|}\sum_{r\in sV\cap KV}(r^{-1}s,1)(s^{-1},a^*)(r,1)\\
			&=&\frac{1}{|V|}\sum_{r\in sV\cap KV}(e,a^*)=\frac{|sV\cap KV|}{|V|}(e,a^*).
		\end{eqnarray*}
		Now, if $s\in K$, then $u_s(s,a)=(e,a^*)$, and if $s\notin U$, then $sV\cap KV=\emptyset$, and so $u_s(s,a)=(e,0)$.	
	\end{proof}
	
	Let $\mathcal{B}=\{B_s\}_{s\in G}$ be the Fell bundle of a system $(A,G,\gamma)$, where $A$ is unital with unit element $1$. Define
	$\Delta_{s,t}:B_{st}\rightarrow B_s\otimes_{min} B_t$ by
	$\Delta_{s,t}(st,a)=(s,1)\otimes (t,a)$. It is easy to see that,
	$$
	(\Delta_{r,s}\otimes\mathrm{id}_{B_t})\Delta_{rs,t}=(\mathrm{id}_{B_r}\otimes\Delta_{s,t})\Delta_{r,st}.
	$$
	Next, if $\pi_{\alpha_1,\beta_1}, \rho_{\alpha_2,\beta_2}\in B(\mathcal{B})$, then
	\begin{eqnarray*}
		(\pi_{\alpha_1,\beta_1}\rho_{\alpha_2,\beta_2})_s(s,a)&=&\langle(\pi_s\otimes\rho_s)(\Delta_{s,s}(s^2,a))\alpha_1\otimes\alpha_2,\beta_1\otimes\beta_2\rangle\\
		&=&	\langle(\pi_s\otimes\rho_s)((s,1)\otimes(s,a))\alpha_1\otimes\alpha_2,\beta_1\otimes\beta_2\rangle\\
		&=&\langle\pi_s(s,1)\alpha_1,\beta_1\rangle\langle\rho_s(s,a)\beta_1,\beta_2\rangle.
	\end{eqnarray*}
	Therefore, $B(\mathcal{B})$ is a Banach algebra. Similarly, one can show that $A(\mathcal{B})$ is a Banach algebra. Moreover,
	for any $\lambda_{\xi,\eta}\in A(\mathcal{B})$ and $\pi_{\alpha,\beta}\in B(\mathcal{B})$,
	$$
	\lambda_{\xi,\eta}\pi_{\alpha,\beta}\in A(\mathcal{B}).
	$$
	Now we have $\ell^1(\mathcal{B})\cong \ell^1(G, A)$, and  $B(\mathcal{B})=C^*(\mathcal{B})^*=(A\rtimes G)^*$. Also, under the canonical correspondence discussed before,
	$P(\mathcal{B})$ is associated  to
	$$
	\{\phi:G\rightarrow A^* : \phi(s)(a)=\varphi(au_s),~\mathrm{for~ some}~\mathrm{positive~ functional~}  \varphi ~\mathrm{on}~ C^*(\mathcal{B})\}
	$$
	In this case, for every $\lambda_{\xi,\eta}\in A(\mathcal{B})$,
	$$\lambda^*_{\xi,\eta,s}(s,a):=\langle\lambda_s(s,a)\xi,\eta\rangle^*=\langle\lambda_{s^{-1}}(s^{-1},a^*)\eta,\xi\rangle=
	\lambda_{\eta,\xi,s^{-1}}(s^{-1},a^*).
	$$
	Therefore, $A(\mathcal{B})$ is a  Banach $*$-algebra. Similarly $B(\mathcal{B})$ is a Banach $*$-algebra.
	\begin{theorem}
		Let $\mathcal{B}$ be the Fell bundle of C*-dynamical system $(A,G,\gamma)$ with  $A$  unital. Then,
		
		\vspace{.2cm}
		${(i)}$ the span closure of $P(\mathcal{B})\cap c_c(G,A^*)$
		is isomorphic to  $A(\mathcal{B})$, as a Banach space.
		
		\vspace{.2cm}
		${(ii)}$ for a positive linear  functional $\psi:B_e\rightarrow\mathbb{C}$ with $\psi(e,1)=1$,
		there is a surjective  bounded linear map $T_\psi:A(\mathcal{B})\rightarrow \overline{\mathrm{span}}(P(\mathcal{B})\cap c_c(G,A^*))$.
	\end{theorem}
	\begin{proof}
		${(i)}$
		Given $\phi\in P(\mathcal{B})\cap c_c(G,A^*)$ with $K=\mathrm{supp}(\phi)$, by Lemma \ref{eym-trick}, there exists
		an element
		$u_{\phi}\in A(\mathcal{B})$ with $(u_\phi)_s(s,a)=(e,a^*)$, for  $s\in K$ and $a\in A$.	
		Define $T:\overline{\mathrm{span}}~P(\mathcal{B})\cap c_c(G,A^*)\rightarrow A(\mathcal{B})$
		by $T(\phi)=u_\phi\phi$.
		
		For $\phi\in  P(\mathcal{B})\cap c_c(G,A^*)$, there are $\xi,\eta\in\ell^2(\mathcal{B})$ and $\pi_{\alpha,\alpha}\in B(\mathcal{B})$ such that
		$u_\phi=\lambda_{\xi,\eta}$ and $\phi=\pi_{\alpha,\alpha}$. Thus, for every $s\in K$,
		\begin{eqnarray*}
			(u_\phi\phi)_s(s,a)&=&(\lambda_{\xi,\eta}\pi_{\alpha,\alpha})_s(s,a)\\
			&=&\langle\lambda_s(s,1)\xi,\eta\rangle\langle\pi_s(s,a)\alpha,\alpha\rangle\\
			&=&(e,1)\langle\pi_s(s,a)\alpha,\alpha\rangle\\
			&=&(e,1)\phi_s(s,a).
		\end{eqnarray*}
		Therefore, $T(\phi)=(e,1)\phi$, and it $T$ is linear with $\|T(\phi)\|=\|\phi\|$.
		Moreover, for $\pi_{\alpha,\beta}$ and $\rho_{\alpha',\beta'}$   in $B(\mathcal{B})$,
		\begin{eqnarray*}
			(T(\pi_{\alpha,\alpha}\rho_{\beta,\beta}))_s(s,a)&=&(e,1)(\pi_{\alpha,\alpha}\rho_{\beta,\beta})_s(s,a)\\
			&=&	(e,1)\langle\pi_s(s,1)\alpha,\alpha\rangle(e,1)\langle\rho_s(s,a)\beta,\beta\rangle\\
			&=&(T(\pi_{\alpha,\alpha})T(\rho_{\beta,\beta}))_s(s,a).
		\end{eqnarray*}
		Thus, $T$ is homomorphism.
		
		$(ii)$ Given $\xi\in c_c(\mathcal{B})$, consider $$\varphi_\xi:\ell^1(\mathcal{B})\rightarrow B_e; \ \varphi_\xi(f)=\overline{\langle\lambda(f)\xi,\xi\rangle}=\overline{\sum_{r \in G}\langle\lambda_r(f_r)\xi,\xi\rangle}.$$ One may extend
		$\varphi_\xi$ to a positive linear map on $C^*(\mathcal{B})$.
		Put $T_\psi(\lambda_{\xi,\xi})(f)=\psi\circ\phi_\xi(f)$. Then,
		$T_\psi(\lambda_{\xi,\xi})$ is a positive linear functional on $C^*(\mathcal{B})$ and
		$T_\psi(\lambda_{\xi,\xi})_s(s,a)=\psi(\overline{\langle\lambda_s(s,a)\xi,\xi\rangle})$. Moreover, by Lemma \ref{positivity}, we have
		$T_\psi(\lambda_{\xi,\xi})\in P(\mathcal{B})\cap c_c(G,A^*).$

		For $\lambda_{\xi,\eta}\in A(\mathcal{B})$
		let $\varphi_{\xi,\eta}:C^*(\mathcal{B})\rightarrow B_e$ be defined by $\varphi_{\xi,\eta}(f)=\overline{\langle\lambda(f)\xi,\eta\rangle}$
		and put $T_\psi(\lambda_{\xi,\eta})=\psi\circ\phi_{\xi,\eta}$.
		There are sequences $(\xi_n)_n$ and $(\eta_n)_n$	in $c_c(\mathcal{B})$ such that
		$\|\lambda_{\xi_n,\eta_n}-\lambda_{\xi,\eta}\|_{A(\mathcal{B})}\rightarrow 0$. Moreover,
		$$
		\lambda_{\xi_n,\eta_n}=\frac{1}{4}\sum_{k=0}^{3}i^k\lambda_{i^k\xi_n+\eta_n,i^k\xi_n+\eta_n}.
		$$
		Therefore, for each $n$,  $T(\lambda_{\xi_n,\eta_n})\in \mathrm{span}~P(\mathcal{B})\cap c_c(G,A^*).$
		
		A straightforward calculation gives
		$\|T_\psi(\lambda_{\xi_n,\eta_n})-T_\psi(\lambda_{\xi,\eta})\|_{B(\mathcal{B})}\rightarrow 0$, that is,
		$T_\psi(\lambda_{\xi,\eta})\in \overline{\mathrm{span}}~P(\mathcal{B})\cap c_c(G,A^*)$.
		
		Finally, if $\phi\in P(\mathcal{B})\cap c_c(G,A^*)$,  by Lemma \ref{eym-trick}, there exists $u_\phi\in A(\mathcal{B})$ such that
		$u_\phi\phi=(e,1)\phi\in A(\mathcal{B})$. Thus, $$T_\psi(u_\phi\phi)_s(s,a)=\psi((e,1)\phi_s(s,a))=\phi_s(s,a),$$
		therefore, $T_\psi$ is onto.
	\end{proof}
	
	We close this section by proving one side of the celebrated Leptin's theorem (cf.,  \cite{Eymard64}) for Fell bundles coming from actions. We leave the inverse implication as an open problem.
	
	\begin{theorem}
		Let $G$ be an amenable discrete group and
		$\mathcal{B}$ be the Fell bundle of a $C^*$-dynamical system $(A,G,\gamma)$ with $A$ unital. Then
		$A(\mathcal{B})$ has a bounded approximate identity.	
	\end{theorem}
	\begin{proof}
		Given $K\subseteq G$ compact and $\epsilon>0$, since $G$ is amenable,
		one may find a finite set $C$ such that $|KC|<(1+\epsilon)|C|$ \cite[Theorem 7.9]{pier}. Next, consider the  collection
		of maps
		$u_{K,\epsilon}\in A(\mathcal{B})\cap c_c(G,\mathcal{L}(\mathcal{B},B_e))$,  defined by
		$$
		u_{K,\epsilon}=\frac{1}{(1+\epsilon)|C|}\lambda_{1_C,1_{KC}}.
		$$
		If $\mathcal{K}(G)$ denote the set of all finite subsets of $G$, with preorder on the set $\mathcal{K}(G)\times\mathbb{R}^+$	defined by
		$(K,\epsilon)<(K',\epsilon')$, whenever $K\subseteq K'$ and $\epsilon>\epsilon'$, then
		$$
		\|u_{K,\epsilon}\|_{A(\mathcal{B})}\leq \frac{1}{(1+\epsilon)|C|}\|1_C\|_2\|1_{KC}\|_2=
		\frac{|KC|^{1/2}}{(1+\epsilon)|C|^{1/2}}\leq\frac{\sqrt{1+\epsilon}}{1+\epsilon}\leq 1.
		$$
		Now let $v=\lambda_{\xi,\eta}\in A(\mathcal{B})\cap c_c(G,\mathcal{L}(\mathcal{B},B_e))$, and set
		$K=\mathrm{supp}(v)$. Then for  $s\in K$,
		\begin{eqnarray*}
			(u_{K,\epsilon})_s(s,a)&=&\frac{1}{(1+\epsilon)|C|}\langle\lambda_s(s,a)1_C,1_{KC}\rangle\\
			&=&	\frac{1}{(1+\epsilon)|C|}\sum_{r \in G}(1_C(s^{-1}r))^*(s^{-1},a^*)1_{KC}(r)\\
			&=&\frac{(e,a^*)}{(1+\epsilon)|C|}\sum_{r \in C}1_{KC}(sr).
		\end{eqnarray*}
		It follows that
		\begin{eqnarray*}
			((u_{K,\epsilon})v)_s(s,a)&=&\langle\lambda_s(s,a)\xi,\lambda_e((u_{K,\epsilon})_s(s,1))\eta\rangle\\
			&=&\langle\lambda_s(s,a)\xi,\lambda_e(\frac{(e,1)}{1+\epsilon})\eta\rangle\\
			&=&\frac{1}{1+\epsilon}v_s(s,a),	
		\end{eqnarray*}
		that is,  $\|u_{K,\epsilon}v-v\|_{A(\mathcal{B})}=\frac{\epsilon}{\epsilon+1}\|v\|_{A(\mathcal{B})}$.
		Thus, $\lim_{(K,\epsilon)}\|u_{K,\epsilon}v-v\|_{A(\mathcal{B})}=0$, for
		$v\in A(\mathcal{B})\cap c_c(G,\mathcal{L}(\mathcal{B},B_e))$.
		
		Next, given $w=\lambda_{\xi,\eta}\in A(\mathcal{B})$, there are sequences $(\xi_n)$ and $(\eta_n)$
		in $c_c(\mathcal{B})$ such that $\|\lambda_{\xi_n,\eta_n}-\lambda_{\xi,\eta}\|_{A(\mathcal{B})}\rightarrow 0$.
		Thus, there is a $v\in A(\mathcal{B})\cap  c_c(G,\mathcal{L}(\mathcal{B},B_e))$ such that $\|w-v\|<\epsilon$. Finally,
		\begin{eqnarray*}
			\left\|u_{K,\epsilon} w-w\right\|_{A(\mathcal{B})} &=&\left\|u_{K, \epsilon} w-u_{K, \epsilon} v+u_{K, \epsilon} v-v+v-w\right\|_{A(\mathcal{B})}\\
			&\leq&\left\|u_{K, \epsilon} w-u_{K, \epsilon} v\right\|_{A(\mathcal{B})}+\left\|u_{K, \epsilon} v-v\right\|_{A(\mathcal{B})}+\|v-w\|_{A(\mathcal{B})} \\
			& \leq&\left\|u_{K, \epsilon}\right\|_{A(\mathcal{B})}\|w-v\|_{A(\mathcal{B})}+\left\|u_{K, \epsilon} v-v\right\|_{A(\mathcal{B})}+\|v-w\|_{A(\mathcal{B})} \\
			& \leq&\left\|u_{K, \epsilon} v-v\right\|_{A(\mathcal{B})}+2 \epsilon,
		\end{eqnarray*}	
		namely, $(u_{K,\epsilon})$ is a left bounded approximate identity for $A(\mathcal{B})$. Since
		$A(\mathcal{B})$ is a Banach $*$-algebra, it follows that $A(\mathcal{B})$ has a two-sided bounded approximate identity.
	\end{proof}

	
	

	
	\bibliographystyle{amsplain}

\begin{thebibliography}{99}
		\bibitem{Ab97}
		F. Abadie-Vicens, Tensor products of Fell bundles over discrete groups (1997).
		arXiv:1301.6883v1.
		
		
		
		\bibitem{BedCon16}
		E. Bedos and R. Conti,
		The Fourier-Stieltjes algebra of a $C^*$-dynamical system,
		{\it Internat. J. Math.} {\bf 27}(6) (2016), 1650050.
		
		
		
		\bibitem{Exel17}
		R. Exel,
		{\it Partial dynamical systems, Fell bundles and applications}, Mathematical Surveys and
		Monographs, vol. 224, Amer. Math. Soc., Providence, RI, 2017.
		
		
		\bibitem{Eymard64}
		P. Eymard,
		L'alg\'{e}bre de Fourier d'un groupe localement compact,
		{\it Bull. Soc. Math. France} {\bf 92} (1964), 181-236.
		
		
		\bibitem{Fell69}
		J. M. G. Fell,
		An extension of Mackey's method to Banach $*$-algebraic bundles,
		\textit{Mem. Am. Math. Soc.}, vol. 90, Amer. Math. Soc., Providence, RI, 1969.
		
		\bibitem{FeDo88}
		J. M. G. Fell and R. S. Doran,
		{\it Representations of ${*}$-algebras, locally compact groups, and Banach $^{*}$-algebraic
			bundles}, vol. I and II, Pure and Applied Mathematics, vol. 125 , Academic Press, Boston, MA, 1988.
		
		
	
		
		\bibitem{KaLa18}
		E. Kaniuth and A.T.-M. Lau,
		{\it Fourier and Fourier Stieltjes algebras on
			locally compact groups}, Mathematical Surveys and Monographs
		Vol. 231, Amer. Math. Soc., Providence, RI, 2018.
		
		
	
		
		\bibitem{M68}		
		G.W. Mackey, {\it Induced representations of groups and quantum mechanics}, Benjamin, Reading, MA, 1968.
		
		\bibitem{pedersen}
		G. K.	Pedersen,  {\it $C\sp{\ast} $-algebras and their automorphism groups}, London Mathematical Society Monographs,
		vol. 14, Academic Press, London, 1979.
		
		\bibitem{pier}
		J. P. Pier, {\it Amenable locally compact groups}, John Wiley and  Sons,  New York, 1984.
		
		\bibitem{Q96}	
		J. Quigg,  Discrete $C^*$-coactions and $C^*$-algebraic bundles, {\it J. Aust. Math. Soc.}, Series A, {\bf 60}(2) (1996), 204-221. 	
		
	
		
	
		
		\bibitem{Tura2000}
		V. Turaev, Homotopy field theory in dimension 3 and group-categories, preprint GT/0005291, 2000.
		
		\bibitem{Vi02}
		A. 	Virelizier,
		Hopf group-coalgebras, \textit{J. Pure Appl. Algebra} \textbf{171} (2002),  75-122.
		
		
		
		
	\end{thebibliography}

\end{document}